\documentclass[reqno]{amsart}
\usepackage{amsthm,amsmath,amsfonts,amssymb,amscd,mathrsfs,graphics,diagrams}
\usepackage{txfonts}
\usepackage{supertabular}
\usepackage{tikz, graphicx}
\usepackage[pdftex,
            pdfauthor={Fan,Francis,Jarvis,Merrell,Ruan},
            pdftitle={D4 Integrable Hierarchies},
            pdfsubject={FJRW theory, Mirror symmetry},
            pdfkeywords={FJRW, Mirror symmetry}]{hyperref}
\newcommand{\Chern}{\operatorname{Ch}} 

\newcommand{\C}{{\mathbb{C}}}
\newcommand{\Deltat}{\widetilde{\Delta}}
\newcommand{\D}{{\mathscr D}}

\newcommand{\Gacp}{{\Gamma'_{\mathrm{cut}}}}
\newcommand{\Gac}{\Gamma_{\mathrm{cut}}}
\newcommand{\Ga}{{\Gamma}}
\newcommand{\Gss}[1]{H^{mid}(\C^N_{J^{#1}}, W_{J^{#1}}^{\infty}, \Q)} 

\newcommand{\Hess}{\operatorname{Hess}}

\newcommand{\J}{{\mathscr J}}

\newcommand{\LL}{{\mathscr L}} 

\newcommand{\MM}{\overline{\M}} 
\newcommand{\M}{{\mathscr M}}

\newcommand{\Q}{{\mathbb{Q}}}

\newcommand{\W}{{\mathscr W}}

\newcommand{\Z}{{\mathbb{Z}}}
\newcommand{\aut}{\operatorname{Aut}}

\newcommand{\bGa}{{\boldsymbol{\Ga}}}

\newcommand{\be}{\mathbf{e}}

\newcommand{\bone}{\be}

\newcommand{\bt}{\mathbf{t}}

\newcommand{\cC}{\mathscr C} 

\newcommand{\chat}{\hat{c}}

\newcommand{\ch}{{\mathscr H}}

\newcommand{\corf}[2]{{\langle #2\rangle^{#1}}} 
\newcommand{\corr}[1]{{\corf{}{#1}}} 

\newcommand{\dsand}{\quad \text{ and }\quad }

\newcommand{\ga}{{\gamma}}
\newcommand{\genj}{\langle J \rangle}

\newcommand{\id}{\boldsymbol{1}}

\newcommand{\ind}{\operatorname{index}}

\newcommand{\milnor}{\mathscr{Q}}

\newcommand{\newxsq}{{X}^2}
\newcommand{\newx}{{X}}
\newcommand{\newy}{{Y}}
\newcommand{\nin}{{\not\in}}

\newcommand{\restr}[2]{\left.#1\right|_{#2}} 

\newcommand{\rhot}{\tilde{\rho}}

\newcommand{\st}{st} 

\newcommand{\unit}{\boldsymbol{1}}

\newcommand{\ve}{{\varepsilon}}

\renewcommand{\Re}{\mathfrak{Re}}
\renewcommand{\hom}{\operatorname{Hom}}
\renewcommand{\l}{\mathbf{\ell}}

\newtheorem{thm}{Theorem}[section]
\newtheorem{lm}[thm]{Lemma}
\newtheorem{sublm}[thm]{Sub-Lemma}
\newtheorem{prop}[thm]{Proposition}
\newtheorem{crl}[thm]{Corollary}

\theoremstyle{definition}
\newtheorem{rem}[thm]{Remark}

\newtheorem{df}[thm]{Definition}

\theoremstyle{remark}

\begin{document}

\date{\today}

\title[$D_4$ Integrable Hierarchies]{Witten's $D_4$ Integrable Hierarchies Conjecture}
\author{Huijun Fan}
\thanks{Partially Supported by NSFC 10401001, NSFC 10321001, and NSFC 10631050}
\address{School of Mathematical Sciences, Peking University, Beijing 100871, China}
\email{fanhj@math.pku.edu.cn}
\author{Amanda Francis}
\address{Department of Mathematics, Brigham Young University, Provo, UT 84602, USA}
\email{amanda@math.byu.edu}
\author{Tyler Jarvis}
\thanks{Partially supported by National Security Agency grant \# H98230-10-1-0181}
\address{Department of Mathematics, Brigham Young University, Provo, UT 84602, USA}
\email{jarvis@math.byu.edu}
\author{Evan Merrell}
\address{Department of Mathematics, Brigham Young University, Provo, UT 84602, USA}
\email{emerrell@math.byu.edu}
\author{Yongbin Ruan}
\thanks{Partially supported by the National Science Foundation and
the Yangtze Center of Mathematics at Sichuan University}
\address{Department of Mathematics, University of Michigan Ann Arbor, MI 48105 U.S.A
and the Yangtz Center of Mathematics at Sichuan University, Chengdu,
China} \email{ruan@umich.edu}


\begin{abstract}
We prove that the total descendant potential functions 
of the theory of Fan-Jarvis-Ruan-Witten for $D_4$ with symmetry group  $\genj$ and $D_4^T$
with symmetry group  $G_{max}$, respectively, are both tau-functions of the $D_4$ Kac-Wakimoto/Drinfeld-Sokolov hierarchy.  This completes the proof, begun in \cite{FJR}, of the Witten Integrable Hierarchies Conjecture for all simple (ADE) singularities.
\end{abstract}

\maketitle

\setcounter{tocdepth}{1}
\tableofcontents

\section{Introduction}

Twenty years ago, Witten proposed a sweeping generalization of his famous conjecture connecting the KdV hierarchy to intersection numbers of the Deligne-Mumford
stack of stable curves. The package includes (i) a first order nonlinear elliptic PDE (Witten equation) to replace the $\bar{\partial}$-equation
in the Landau-Ginzburg setting; (ii) a conjecture  asserting that the total potential function of an $ADE$-Landau-Ginzburg orbifold $(W, \langle J \rangle)$
is a tau-function of the corresponding integrable hierarchy. Here, $J=(\exp(2\pi i q_1), \dots, \exp(2\pi i q_N))$ is the so-called \emph{exponential grading operator}, where  the $q_i$ are the weights of the  quasi-homogeneous polynomial $W$. In a series of papers, Fan-Jarvis-Ruan have constructed the theory
Witten expected and solved the conjecture for the D and E-cases except $D_4$, while the $A_n$-case was solved earlier by Faber-Shadrin-Zvonkine \cite{FSZ}.

But the story is still incomplete, and a resolution of the conjecture for the $D_4$ and $D_4^T$ cases is needed. This is the main purpose of this article.
When we say $D_4$, we mean the singularity $D_4:=x^3+xy^2$.  Recall that its exponential grading operator is $J=(\exp(2\pi i/3), \exp(2\pi i/3)$. 
Its dual or transpose singularity 
is $D^T_4=x^3y+y^2$,
 and the exponential grading operator of $D_4^T$ is $J^T=(\exp(\pi/3), -1)$. 
Let ${\D}^{FJRW}_{D_4,\genj}$ and  ${\D}^{FJRW}_{D^T_4,G_{max}}$ 
 (see notation in next section), denote the respective 
total descendant potential functions 
of the theory of Fan-Jarvis-Ruan-Witten for $D_4$ and $D_4^T$
  with symmetry groups  $\langle J\rangle$, and $G_{max}$, respectively.
The main theorem of this paper is the following.

\begin{thm}\label{thm:main}
\ 
\begin{enumerate}
\item ${\D}^{FJRW}_{D_4, \genj}$ is a tau-function of the $D_4$ Kac-Wakimoto/Drinfeld-Sokolov hierarchy.
\item ${\D}^{FJRW}_{D^T_4, G_{max}}$ is a tau-function of the $D_4$ Kac-Wakimoto/Drinfeld-Sokolov hierarchy.
\end{enumerate}
\end{thm}

To explain the difficulty in the case of $D_4$ and $D^T_4$ with symmetry group $\langle J\rangle$ and $G_{max}$, respectively, let's briefly review the proof in \cite{FJR} of the conjecture for the other simple singularities. The idea of the proof is to identify the theory constructed
by Fan, Jarvis, and Ruan for a given singularity $W$ with the Saito-Givental theory of a related singularity $W'$. The later has been shown to give a tau-function of the corresponding hierarchy
\cite{FGM}. The proof consists of two steps: (1) prove two reconstruction theorems to completely determine Fan-Jarvis-Ruan-Witten theory and the Saito-Givental theory
from the pairing, the three-point correlators, and certain {special} four-point correlators; 

(2) compute the three-point and special four-point correlators 
explicitly and match them. The reconstruction theorems apply to $D_4$ and its dual singularity $D^T_4$ as well. But the computation of the 
{special} four-point correlators is much more difficult for the cases of $D_4$ with symmetry group $\langle J\rangle$ and $D_4^T$ with symmetry group $G_{max}$ than it is for
 all the other ADE singularities.\footnote{Note that for  $D_4^T$ the symmetry group $\langle J^T\rangle$ 
is equal to the maximal admissible symmetry group $G_{max}$ for this theory, but for $D_4$ the group $\langle J\rangle$ has index $2$ in the maximal admissible symmetry group of $D_4$.}
 
In all cases except these two, the relevant insertions
are what we call \emph{narrow} (called \emph{Neveu-Schwarz} in \cite{FJR, JKV}). An early lemma of Witten asserts that the Witten equation with exclusively narrow insertions has only the zero solution.  In such a case, the problem under study can be reduced to an algebro-geometric one and the computation can be carried out in a straightforward manner. 

But these two cases of $D_4$ and $D^T_4$ are 
entirely different, and instead of being narrow, some of the insertions
are \emph{broad} (called \emph{Ramond} in \cite{FJR}). Therefore, the  algebro-geometric reduction does not apply.  The problem under study is a PDE-problem and we do not yet have techniques to solve it explicitly in the case of broad insertions. 

In the case of $D_4^T$ we must compute two special correlators.  One of the two correlators has only narrow insertions and its value was already computed in \cite{FJR}.  The final result follows (specifically, all the A-side correlators match those on the B-side) if and only if the other special correlator vanishes.  Although the correlator has broad insertions, its vanishing is fairly simple to prove.

But in the case of $D_4$, those techniques do not work.  Again we need to prove that one special four-point correlator vanishes and that the other does not vanish.  The vanishing part is again simple, but the remaining special four-point correlator has broad insertions and cannot be computed easily.
To prove that it does not vanish, we need a new idea, namely, to look for other correlators with only narrow insertions and use those to reconstruct the correlator with broad insertions. Indeed, there is a unique primary correlator with purely narrow insertions---the unique
highest-point (seven) correlator for the $D_4$-theory.  

We completely describe the genus-zero primary potential in terms of the one special four-point correlator.  We further show that this correlator vanishes if and only if the the unique seven-point correlator vanishes. This part of the argument uses the full strength of the WDVV-equation.  Then, with some work, we compute the seven-point correlator using algebro-geometric methods, and show that it is not zero, as required. 

\subsection{Acknowledgments}
We wish to give special thanks to Carel Faber and Drew Johnson for providing their code \cite{FaMgnF} and \cite{Jo}, respectively, for computing intersections of divisors on the stack of stable curves.  Thanks also to Youjin Zhang and Di Yang for helping us identify errors in an earlier draft of this paper.

The third author also wishes to thank the Max Planck Institute f\"ur Mathematik in Bonn and the Institute Henri Poincar\'e, where much of the work on this paper was completed.

\section{Background and Notation}

 In this section, we briefly review the theory of \cite{FJR} for $D_4$ and $D^T_4$ to set up the notation.
  Recall that for each quasi-homogeneous polynomial $W\in \C[x_1,\dots,x_N]$ of weights $q_1,\dots,q_N$, the \emph{exponential grading operator} $J:=(\exp(2\pi i q_1),\dots,\exp(2\pi i q_N))$ lies in the group $\aut(W)$ of diagonal
matrices $\gamma$ such that $W(\gamma x)=W(x)$.  Given any non-degenerate $W$ and any subgroup $G\le \aut(W)$ such that $J\in G$, the results of \cite{FJR} provide a cohomological field theory $(\ch_W^G, \langle\, , \rangle^W, \{\Lambda^W_{g,k}\}, \id)$
with flat identity.

The state space $\ch_W^G$ is defined as follows.  For each $\gamma\in G$, let $\C^N_{\gamma}$ be the fixed point set
of $\gamma$, let $N_\ga$ denote the dimension of $\C^N_{\gamma}$, and let $W_{\gamma}=\restr{W}{\C^N_{\gamma}}$. Let
$\ch_{\gamma,G}$ be the $G$-invariants of the middle-dimensional
relative cohomology 
    $$\ch_{\gamma}= H^{mid}(\C^N_{\gamma}, (\Re W)^{-1}(M, \infty),
\C)^{G}$$ of $\C^N_{\gamma}$ for $M>\!\!>0$, as described in Section 3
 of \cite{FJR}. Wall's theorem \cite{Wa1, Wa2} states that the cohomology group $H^{mid}(\C^N_{\gamma}, (\Re W)^{-1}(M, \infty),
\C)$ is isomorphic, as a graded $G$-module, to the space $\milnor_{W_\gamma} \omega_\gamma$, where $\milnor_{W_\gamma}$ is the Milnor ring (local algebra) of $W_{\ga}$ and $\omega_{\ga} = dx_{i_1}\wedge \dots\wedge dx_{i_{N_\ga}}$ is the canonical volume form on $\C^N_\ga$.  For computational purposes, it is generally much easier to use the Milnor ring than the middle cohomology, so we will assume from now on that an isomorphism 
\begin{equation}\label{eq:iso}
\ch_{\gamma}= H^{mid}(\C^N_{\gamma}, (\Re W)^{-1}(M, \infty),
\C)^{G} \cong \left(\milnor_{W_\ga}\omega_{\ga}\right)^G
\end{equation} has been chosen for each $\ga$, once and for all.

The \emph{state space} of our theory is the sum
    $$\ch_{W,G}=\bigoplus_{\gamma\in G} \ch_{\gamma}.$$ The state space
$\ch_{W,G}$ admits a
grading and a non-degenerate
pairing $\langle\,\rangle^W$.  The pairing is essentially the residue pairing on the underlying Milnor rings $\milnor_{W_\ga}$, but with the elements of $\ch_\ga$ only pairing with elements of $\ch_{\ga^{-1}} \cong \ch_\ga$.  The grading is more subtle, as we now describe.

\begin{df}The \emph{central
charge} of the singularity $W_\gamma$ is denoted $\chat_\gamma$:
\begin{equation*}
\chat_\gamma:=\sum_{i:\Theta^\gamma_i = 0}
(1-2q_i)\glossary{chatgamma@$\chat_{\gamma}$ & The central charge
of
  the singularity $W_\gamma$}.
\end{equation*}
\end{df}
\begin{df} Suppose that $\gamma=(e^{2\pi
    i \Theta^\gamma_1}, \dots, e^{2\pi i\Theta^\gamma_{N}})$ for rational numbers
    $0\leq \Theta^\gamma_i<1$.

    We define the \emph{degree shifting number}
\begin{align*}\glossary{iotagamma@$\iota_{\gamma}$ & The degree-shifting number for $\ch_\gamma$}
 \iota_{\gamma}&=\sum_i (\Theta^{\gamma}_i
 -q_i) 
 \\
 &=\frac{\hat{c}_W-N_\gamma}{2}+\sum_{i:\Theta^{\gamma}_i\neq 0}
 (\Theta^{\gamma}_i-1/2)
 \\
 &=\frac{\hat{c}_{\gamma}-N_\gamma}{2} +
 \sum_{i:\Theta^{\gamma}_i\neq0}(\Theta^{\gamma}_i-q_i).
\end{align*}
For a class $\alpha \in \ch_{\gamma}$, we define

\begin{equation*}
\deg_\C(\alpha)=\frac{1}{2}\deg(\alpha)+ \iota_{\gamma} = \frac{1}{2}N_\ga+ \iota_{\gamma}.
\glossary{degw@$\deg_\C(\alpha)$ & The orbifold degree
$\deg(\alpha) + 2\iota_{\gamma}$ of the (homogeneous) class
$\alpha \in \ch_{\gamma}$}
\end{equation*}
      \end{df}

The classes $\Lambda^W_{g,k} \in \hom(\ch_W^{\otimes k}, H^*(\MM_{g,k}))$ satisfy the usual axioms of a cohomological field theory with flat identity, including symmetry, and the composition axioms (see \cite[\S4.2]{FJR} for details).  Here $\MM_{g,k}$ is the stack of stable, $k$-pointed, genus-$g$ curves.

When an element $\ga\in G$ fixes only $0\in\C^N$, we say that the 
sector $\ch_\ga$ is \emph{narrow} (called Neveu-Schwarz in \cite{JKV}), and when the element $\ga$ fixes a non-trivial subspace $\C^{N_\ga}\subset \C^N$ with $N_\ga >0$, we say that the sector $\ch_\ga$ is \emph{broad} (called \emph{Ramond} in \cite{JKV}). For each narrow sector $\ch_\ga$, the isomorphism (\ref{eq:iso}) defines an element $\bone_\ga$ as the image of $1\in \C \cong \ch_\ga$ in $\ch_\ga$.  The exponential grading operator $J$ is narrow, and the isomorphism (\ref{eq:iso}) can be chosen so that the element $\bone_J$ is the flat identity $\id$.  The elements $\bone_\ga$ also  play an important role in one additional axiom that allows us to compute the classes $\Lambda^W_{g,k}$ in some special cases, namely the \emph{Concavity Axiom}, which we now briefly review.  

As mentioned above, in the case that all the insertions are narrow, the construction of the classes $\Lambda^W_{g,k}$  reduces to an algebro-geometric problem on the moduli of $W$-curves $\W_{W,g,k}$.  The universal $W$-structure on  the universal $W$-curve $\cC\rTo^{\pi}\W_{W,g,k}$ corresponds to a choice of orbifold line bundles $\LL_1,\dots,\LL_N$ on $\cC$ which are roots of the log-canonical bundle $K_{\log}$, with certain additional properties described in \cite{FJR}.  In the special case that these bundles are \emph{concave}, i.e., when the pushforward $\pi_*\left(\bigoplus_{i=1}^N \LL_i\right) = 0$, then the class 
$\Lambda^W_{g,k}(\bone_{\ga_1},\dots,\bone_{\ga_k})$ is given by 
$$\Lambda^W_{g,k}(\bone_{\ga_1},\dots,\bone_{\ga_k}) = \frac{|G|^g}{\deg(\st)} st_*\left( c_{top} \left(R^1\pi_* \bigoplus_{i=1}^N \LL_i\right)^*\right),$$
where $\st:\W_{W,g,k} \rTo \MM_{g,k}$ is the canonical morphism from the stack of $W$-curves to the stack of stable curves, where $c_{top}$ is the top Chern class, or Euler class, and where $\left(R^1\pi_* \bigoplus_{i=1}^N \LL_i\right)^*$ is the dual of the bundle $R^1\pi_* \bigoplus_{i=1}^N \LL_i$.

Another essential property of the classes $\Lambda^W_{g,k}$ is that the cohomology classes have complex dimension equal to 
$D+\sum \frac12N_\ga$, 
where $-D$ is the sum of the indices of the $W$-structure bundles:
\begin{equation*}
D:=-\sum_{i=1}^N \ind(\LL_i)=\chat_W(g-1)+\sum_{j=1}^k \iota_{\gamma_j},
\end{equation*}
and $N_\ga$ is the complex dimension of the fixed locus $(\C^N)^{\ga}$ of $\ga$.

\begin{df}\label{df:correlator}
For the cohomological field theory $(\ch_W^G, \langle\, , \rangle^W, \{\Lambda^W_{g,k}\}, \id)$, we define correlators in the standard manner, as 
   $$\langle \tau_{l_1}(\alpha_1), \dots, \tau_{l_k}(\alpha_k)\rangle ^W_g:=\int_{\left[\MM_{g,k}\right]}
   \Lambda^W_{g,k}(\alpha_1, \dots,
   \alpha_k)\prod_{i=1}^k {\psi}^{l_i}_i.$$\glossary{<@$\langle \tau_{l_1}, \dots, \tau_{l_k}\rangle.$ & The correlators of the $W$-theory}
   \end{df}

The genus-zero, three-point correlators of any cohomological field theory define an associative multiplication $\star$ on the state space, making the state space into a Frobenius algebra.  In our case the new (degree-shifted) grading on $\ch_{W,G}$ is also compatible with the multiplication, making $\ch_{W,G},\star$ into a graded Frobenius algebra.

The main selection rules in this theory are 
\begin{itemize}
\item $\langle \alpha_1,\dots,\alpha_k\rangle_g^W = 0$ unless 
\begin{equation}\label{eq:selD}
\hat{c}(g-1)  + \sum_{i=1}^k \deg_\C(\alpha_i)=\dim(\MM_{g,k}) = 3g-3+k,
\end{equation}
\item $\langle \alpha_1,\dots,\alpha_k\rangle_g^W = 0$ unless for each of the coarse line bundles $|\LL_j|$ underlying the $W$-structure, the degree 
\begin{equation}\label{eq:selInt}
\deg(|\LL_j|) = q_j(2g-2+k) - \sum_{\ell=1}^k \Theta_j^{\ga_\ell}
\end{equation}
 is integral.
\end{itemize}
Furthermore, the selection rule~(\ref{eq:selInt}) must hold on each irreducible component of a stable curve.

\begin{rem}\label{rem:invariance}
When the group $G\le\aut(W)$ is not equal to the maximal group $\aut(W)$ of diagonal symmetries, then $\aut(W)$ acts on the A-model state space, and the correlators are all invariant under this action.  In particular, this property forces many correlators to vanish.
\end{rem}

\begin{df}
Let $\{\alpha_0,\dots,\alpha_s\}$ be a basis of the state space
$\ch_W$ such that $\alpha_0 =  \unit:=\bone_{J}$, and let $\bt =
(\bt_0,\bt_1,\dots)$ with $\bt_l =
(t_l^{\alpha_0},t_l^{\alpha_1},\dots, t_l^{\alpha_s})$ be formal
variables.  Denote by $\Phi^W(\bt) \in \lambda^{-2}\C[[\bt,\lambda]]$
the (large phase space) potential \glossary{Phi @$\Phi^W(\bt)$ & The
  large phase-space potential of $W$-theory} of the theory:
$$\Phi^W(\bt):=\sum_{g\ge0} \Phi^W_g(\bt) := \sum_{g\ge 0} \lambda^{2g-2} \sum_{k}\frac{1}{k!} \sum_{l_1,\dots,l_k}\sum_{\alpha_{i_1},\dots,\alpha_{i_k}} \langle \tau_{l_1}(\alpha_{i_1}) \cdots \tau_{l_k}(\alpha_{i_k})\rangle^W_{g} t_{l_1}^{\alpha_{i_1}}\cdots t_{l_k}^{\alpha_{i_k}}.$$
\end{df}

\begin{thm}[See \protect{\cite[Thm 4.2.8]{FJR}}]\label{thm:StringDilaton}
The potential $\Phi^W(\bt)$ satisfies analogues of the string and
dilaton equations and the topological recursion relations.
\end{thm}

\section{Frobenius Algebras}

\subsection{Frobenius Algebra for the B-model of $D_4$}

The Frobenius Algebra for $D_4$ is the local algebra (Milnor ring) $$\milnor_{D_4}  = \C[X,Y]/(3X^2 + Y^2, 2XY)$$ with the residue pairing.  That is, if $$fg = \alpha \frac{\Hess D_4}{\mu} + \text{lower order terms} = \alpha \frac{24X^2}{4} + \text{l.o.t} = \alpha \frac{-8Y^2}{4} + \text{l.o.t},$$ then
$$\corr{f,g} = \alpha.$$

We have $\chat = 2/3$ and $\deg_\C(X) = \deg_\C(Y) = 1/3$.

\subsection{Frobenius Algebra for the A-model of $(D_4^T, G_{max})$}

The singularity $D_4^T$ is defined by the polynomial $W=x^3y+y^2$, and the exponential grading operator $J$ is
$$J=(\xi, \xi^3),\quad \text{ where $\xi=\exp({{2\pi i}/{6}})$}.$$ The element $J$ has order $6$ and generates the maximal group $G_{max}$ of diagonal symmetries of $D_4^T$.  

For $k$ =1, 3, and 5, $\left(D_4^T\right)_{J^k} = 0$, so $\milnor_{\left(D_4^T\right)_J^k}\omega_{J^k} = \langle 1 \rangle$.  
For $k =2$ and 4, $\left(D_4^T\right)_{J^k}= y^2$, so $\milnor_{\left(D_4^T\right)_J^k}\omega_{J^k} = \langle 1dy \rangle$. However, $1 dy$ is not fixed by the action of the group element $J$.
For $k =0$, $\left(D_4^T\right)_{J^k} = D_4^T$, so $\milnor_{\left(D_4^T\right)_J^k}\omega_{J^k} = \langle 1 dx\!\wedge\! dy, x dx\!\wedge\! dy, x^2 dx\!\wedge\! dy, y dx\!\wedge\! dy \rangle$. Only the element $x^2 dx\!\wedge\! dy$ is fixed by the action of $J$, and therefore by all of $G^{max}$.
Thus the set $\{\bone_1,x^2{\bone}_0,{\bone}_3, \bone_5\}$ is a basis of ${\ch}_{D_4^T}^{G_{max}}$, where
$$\bone_0:=dx\!\wedge\! dy \in \Gss{n} = \Gss{0},$$
$$\bone_{1}= 1 \in \Gss{1} \cong \C$$
$$\bone_{{3}}= 1 \in \Gss{3}\cong \C$$
$$\bone_{{5}}= 1 \in \Gss{5}\cong \C.$$

The Frobenius algebra structure on $\ch_{D_4^T,G_{max}}$ was determined in \cite[\S5.2.4]{FJR}, where it was shown that the structure is given by an isomorphism to $\milnor_{D_4}$ as follows:
$\C[X,Y]/(3X^2+Y^2, 2XY) \rTo \ch_{D_4^T,G_{max}}$, with 
\begin{align*}
1\mapsto \bone_1 & \quad Y \mapsto 3 \alpha x^2\bone_0 \\
X \mapsto \alpha \bone_3 & \quad X^2 \mapsto \alpha^2\bone_5,
\end{align*}
where $\alpha^2 = 1/6$.

As in the previous case, we have $\chat = 2/3$ and $\deg_\C(X) = \deg_\C(Y) = 1/3$.

\subsection{Frobenius Algebra for the A-model of $(D_4, \genj)$}

The singularity $D_4$ is defined by the polynomial $W=x^3+xy^2$, and the exponential grading operator $J$ is
$$J=(\xi^2, \xi^2),\quad \text{ where $\xi=\exp({{2\pi i}/{6}})$}.$$ The element $J$ has order $3$ in the group  $\aut({D_{4}}) = \langle \lambda \rangle \cong \Z/6\Z$,
where $\lambda:=
(\xi^{-2},\xi)$.   

For $k$ =1 and 2, $\left(D_4\right)_{J^k} = 0$, so $\milnor_{\left(D_4\right)_J^k}\omega_{J^k} = \langle 1 \rangle$.  
For $k =0$, $\left(D_4\right)_{J^k} = D_4$, so $\milnor_{\left(D_4\right)_J^k}\omega_{J^k} = \langle 1 dx\!\wedge\! dy, x dx\!\wedge\! dy, x^2 dx\!\wedge\! dy, y dx\!\wedge\! dy \rangle$. Only the elements $x dx\!\wedge\! dy$ and $y dx\!\wedge\! dy$ are fixed by the action of $J$, and therefore by all of $G$.
Thus the set  $\{\id,x{\bone}_0,y{\bone}_0,{\bone}_2\}$ is a basis of ${\ch}_{D_4}^{\langle J\rangle}$, where
$$\bone_0:=dx\!\wedge\! dy \in \Gss{n} = \Gss{0},$$
$$\id= \bone_{1}= 1 \in \Gss{1} \cong \C$$
$$\bone_{{2}}= 1 \in \Gss{2}\cong \C.$$

The Frobenius algebra structure on $\ch_{D_4,\langle J\rangle}$ was determined in \cite[\S5.2.4]{FJR}, where it was shown that the primitive classes are
$x{\bone}_0,y{\bone}_0$.    The $D_4$ case is unusual because the primitive classes are all broad.  

It is shown in \cite[\S5.2.4]{FJR} (and is straightforward to check directly) that the pairing is given by 
\begin{equation*}
\langle x\bone_0,x\bone_0 \rangle =\frac{1}{6}, \langle x\bone_0,y\bone_0 \rangle =0, \langle y\bone_0,y\bone_0 \rangle
=\frac{-1}{2}.
\end{equation*}

We note that the ``obvious'' isomorphism of Frobenius algebras $\ch_{D_4,\langle J\rangle} \rTo \milnor_{D_4} := \C[\newx,\newy]/(3\newx^2+\newy^2,2\newx\newy)$ given in \cite[\S5.2.4]{FJR} does not turn out to be the right one for our purposes in this paper.  Instead, we will use the following isomorphism of Frobenius algebras:
\begin{align*}
\bone_1 \mapsto \id & \quad & x{\bone}_0 \mapsto \frac{\newy}{\sqrt{-3}}  \\
\bone_2 \mapsto 6 \newxsq& \quad & y{\bone}_0 \mapsto \sqrt{-3}\newx.
\end{align*}

Via this isomorphism we also have,  
\begin{equation*}
\langle \newx,\newx \rangle =\frac{1}{6},\ \quad  \langle \newx,\newy \rangle =0,\ \quad  \langle \newy,\newy \rangle
=\frac{-1}{2}.
\end{equation*}

As in the previous two cases, we have $\chat = 2/3$ and $\deg_\C(X) = \deg_\C(Y) = 1/3$.

\begin{df}
Hereafter, we will take $\{\id,\newx,\newy,\newxsq\}$
as the standard basis for all three cases: the A-models ${\ch}_{D^T_4,G_{max}}$, and ${\ch}_{D_4,\langle J\rangle}$, and the B-model $\milnor_{D_4}$.
The dual basis is
$\{6\newxsq,6\newx,-2\newy,6\id\}$.
\end{df}

\section{Shared Properties of the Correlators and  Potentials}

In this section we will discuss some properties that all three of our theories share, including various relations among the primary correlators and the potentials.

Recall the selection rules in \ref{eq:selD} and \ref{eq:selInt}.  In all three of our cases, \ref{eq:selD} can be restated as 
\begin{equation}\label{eq:seldeg}
\corr{\varkappa_1,\dots,\varkappa_k} = 0 \qquad \text{unless $\sum \deg_\C(\varkappa_i) = \frac{3k-7}{3}$}.
\end{equation}
Moreover, we have 
\begin{equation}\label{eq:novac}
\corr{\varkappa_1,\dots,\varkappa_{k-1},\id} = 0 \qquad \text{unless $k=3$, }
\end{equation}
and
\begin{equation}\label{eq:pairingax}
\corr{\varkappa_1, \varkappa_2, \id} = \langle \varkappa_1, \varkappa_2 \rangle.
\end{equation}

From these selection rules, we can see immediately that in all three theories, all the three-point correlators vanish except  $\corr{X,X,\id} = \corr{X^2,\id,\id}=\frac16$, and $\corr{Y,Y,\id}=-\frac12$.

\begin{thm}\label{thm:frptvanish}
In all three cases (Saito for $D_4$,  FJRW for $(D^T_4,G_{max})$, and FJRW for $(D_4,\genj)$) the four-point correlators $\corr{Y,Y,Y,X^2}$ and $\corr{Y,X,X,X^2}$ vanish.
\end{thm}
\begin{proof}
For the Saito Frobenius manifold this was computed in \cite[Prop 6.4.4]{FJR}.  

For the FJRW theory of $(D_4^T, G_{max})$,  the selection rule of Equation~(\ref{eq:selInt}) can be applied; namely, we will show that for the two correlators in question, the degree of the line bundle $|\LL_x|$ is not integral, and hence the correlator must vanish.  To do this, we need only the fact that $\Theta_x^{(J^T)^a} = a/6$, which is straightforward to check. We have for $\corr{Y,X,X,X^2}$
\begin{align*}
 \deg(|\LL_x|) 
 &= q_x(2g-2+k) - \sum_{\ell = 1}^k  \Theta_x^{\gamma_\ell}\\
 & = \frac16(-2 + 4) - (3/6 +3/6 + 0/6 + 5/6)  = -\frac32 \ \nin \Z.
\end{align*}
This shows that the correlator $\langle Y,X,X, X^2\rangle$ must vanish.
A similar computation show that the correlator $\langle Y,Y,Y, X^2\rangle$ must also vanish.

Finally, for the FJRW theory of $(D_4,\genj)$ we note that the maximal group of symmetries contains an element $\lambda = (\xi^{-2},\xi)$ which is not contained in $\genj$.  
A simple computation gives the action of $\lambda$
on the basis:

\begin{eqnarray*}
\lambda \id = \id, & \lambda \newx = \newx, & \lambda \newxsq = \newxsq, \\
{\text{but }} & \lambda \newy = -\newy.\notag
\end{eqnarray*}
As observed in Remark~\ref{rem:invariance}, all correlators must be $\lambda$-invariant.  This forces all correlators with an odd number of $Y$-insertions to vanish, as desired.
\end{proof}

\begin{prop}
For all three of our cases, the only genus-zero primary correlators that are not \emph{a priori} forced to vanish by the selection rules (\ref{eq:seldeg}) and (\ref{eq:novac}) are
$$ \langle \id,\id,\newxsq \rangle, 
\langle \id,\newy,\newy \rangle, 
\langle \id,\newx,\newx \rangle, 
$$
$$
\langle \newy,\newy,\newx,\newxsq \rangle, 
\langle \newx,\newx,\newx,\newxsq \rangle, 
$$
$$ 
\langle \newy,\newy,\newxsq,\newxsq,\newxsq \rangle,
\langle \newy,\newx,\newxsq,\newxsq,\newxsq \rangle, 
\langle \newx,\newx,\newxsq,\newxsq,\newxsq \rangle,$$
$$ 
\langle \newy,\newxsq,\newxsq,\newxsq,\newxsq,\newxsq\rangle, 
\langle \newx,\newxsq,\newxsq,\newxsq,\newxsq,\newxsq \rangle,
$$
$$\langle
\newxsq,\newxsq,\newxsq,\newxsq,\newxsq,\newxsq,\newxsq
\rangle.$$
\end{prop}
\begin{proof}
For example, the only nonzero genus-zero three-point correlators must satisfy Equation \ref{eq:seldeg}: 
$ \displaystyle{\sum_{i = 1}^3 \deg_{\C}(\alpha_i) = \frac23}$. 
Our choices are $\langle \id,\id,\newxsq \rangle$, $\langle \id,\newx,\newx \rangle$, $\langle \id,\newx,\newy \rangle$, and $\langle \id,\newy,\newy \rangle$.  Apply Equation \ref{eq:pairingax}. 

The other $k$-point correlators can be found using the selection rules \ref{eq:seldeg} and \ref{eq:novac} with the results of Theorem \ref{thm:frptvanish}.
\end{proof}

A key tool that we will need in this paper is the following generalization of the \emph{Reconstruction Lemma} of \cite{FJR}.

\begin{lm}\label{lm:newReconst}
For both the A- and B-models, given any genus-zero, $k$-point correlator of the form $\corf{}{\ga_1,\dots,\ga_{k-3},\alpha, \beta,\ve\star\phi}$, choose a basis $\{\delta_i\}$ of $\ch_{W,G}$ such that $\delta_0=\ve\star\phi$, and let
   $\delta'_i$ be the dual basis with respect to the pairing (i.e.,
   $\langle\delta_i, \delta'_j\rangle=\delta_{ij}$).

Let $K = \{\ga_1,\dots,\ga_{k-3}\}$.
The correlator $\corf{}{\ga_k \in K,\alpha, \beta,\ve\star\phi}$ can be rewritten as
\begin{align*}
\corf{}{\ga_k \in K,\alpha, \beta,\ve\star\phi}
=&\sum_{{K=I\sqcup J}}\sum_{\ell} c_{I,J}\corf{}{\gamma_{i}\in I, \alpha, \ve ,\delta_\ell}\corf{}
{\delta'_\ell, \phi, \beta,   \gamma_{j}\in J} \\
&\quad -\sum_{\substack{K=I\sqcup J\\ J\neq \emptyset}}\sum_\ell c_{I,J} \corf{}{\gamma_{i}\in I, \alpha,
\beta ,\delta_\ell}\corf{}{\delta'_\ell, \phi, \ve,  \gamma_{j}\in J},\notag
\end{align*}
where 
\begin{equation*}
c_{I,J} = \frac{\prod n_K(\xi_k)!}{\prod n_I(\xi_i)! \prod n_J(\xi_j)!}
\end{equation*}
are integer coefficients.  Here, $n_X(\xi_x)$ refers to the number of elements equal to $\xi_x$ in the tuple $X$. The product $\prod n_X(\xi_x)! $ is taken over all distinct elements $\xi_x$ in $X$.
\end{lm}

\begin{proof}
Let $\mathcal{B}$ be a basis for $\ch_{W,G}$ which contains the element $\ve \star \phi$, and   let $\delta_0$ the corresponding element in the dual basis.  Summing over $\mathcal{B}$ in the definition of the product gives $\ve \star \phi = \langle \ve, \phi, \delta_0\rangle \ve \star \phi$, so
\[
\corf{}{\ga_k \in K,\alpha, \beta,\ve\star\phi} = \corf{}{\ve,\phi, \delta_0} \corf{}{\ga_k \in K,\alpha, \beta,\ve\star\phi}
\]
The WDVV equations show that 
\[
\sum_{I \sqcup J = K} \sum_l c_{I,J}\corf{}{\ga_k \in K,\alpha, \beta,\delta_l}  \corf{}{\ga_k \in K,\ve, \phi,\delta'_l} 
=\sum_{I \sqcup J = K} \sum_l c_{I,J}\corf{}{\ga_k \in K,\alpha, \ve,\delta_l}  \corf{}{\ga_k \in K,\beta, \phi,\delta'_l} 
\]
So, 
\[
\begin{array}{c}
\corf{}{\ga_k \in K,\alpha, \beta,\ve\star\phi} = \corf{}{\ve,\phi, \delta_0} \corf{}{\ga_k \in K,\alpha, \beta,\ve\star\phi}\\
=\displaystyle{
\sum_{\substack{I \sqcup J = K\\l}}  c_{I,J}\corf{}{\ga_k \in K,\alpha, \ve,\delta_l}  \corf{}{\ga_k \in K,\beta, \phi,\delta'_l} 
-\sum_{\substack{I \sqcup J = K\\J \neq \emptyset \\l}} c_{I,J}\corf{}{\ga_k \in K,\alpha, \beta,\delta_l}  \corf{}{\ga_k \in K,\ve, \phi,\delta'_l} }
\end{array}
\]
 \end{proof}

Using the reconstruction lemma and a genus reduction argument, one can show the following:
\begin{thm}\cite[Thm 6.2.10]{FJR} \label{thm:allfromfour}
In both the A- and B-model, the total descendant potential function $\D$ for our three cases (Saito for $D_4$,  FJRW for $D^T_4,G_{max}$, and FJRW for $D_4,\genj$)  are completely determined by the pairing, the three-point correlators (i.e., the Frobenius algebra structure) and the four-point correlators.
\end{thm}

\begin{thm}\label{thm:potential}
Denote the four-point correlator $$ a:=\langle \newx,\newx,\newx,\newxsq \rangle,\ $$
For all three cases (Saito for $D_4$,  FJRW for $D^T_4,G_{max}$, and FJRW for $D_4,\genj$), the primary genus-zero potential has the following form

\begin{align*}
\Phi(\bt) 
&=\frac{1}{12}t_X^2t_{\id}-\frac{1}{4}t_Y^2t_{\id} + \frac{1}{12}t_{X^2}t_{\id}^2 
+ \frac{a}{6} t_X^3t_{X^2} + \frac{3a}{2} t_Xt_Y^2t_{X^2}
+ a^2 t_X^2t_{X^2}^3 - 3a^2 t_Y^2t_{X^2}^3 
+ \frac{36a^4}{35}t_{X^2}^7.
\end{align*}

\end{thm}
\begin{proof}
This follows by running the reconstruction lemma ``in reverse."
First we apply the Reconstruction Lemma~(\ref{lm:newReconst}) to the four-point correlator
\begin{equation*}
\begin{split}
\langle \newx,\newy,\newy,\newxsq\rangle=  
\sum_{\delta}c_{\emptyset, \{X\}}\langle \newx,\newy,\delta\rangle\langle \delta',\newx,\newy,\newx\rangle + \\ \sum_{\delta}c_{ \{X\}, \emptyset}\langle \newx, \newy,\newx,\delta\rangle\langle \delta',\newy,\newx\rangle 
-\sum_{\delta}c_{ \{X\}, \emptyset}\langle \newx, \newx,\newx,\delta\rangle\langle \delta',\newy,\newy\rangle,  
\end{split}
\end{equation*}
where the sums run over a basis $\delta$ of the state space, and $\delta'$ is the element dual to $\delta$. Notice that $ c_{\emptyset, \{X\}}=c_{ \{X\}, \emptyset} = 1$. This yields  
\begin{equation*}
\langle \newx,\newy,\newy,\newxsq\rangle = -\langle \newx, \newx,\newx,\newxsq\rangle\langle 6,\newy,\newy\rangle   
= 3a.
\end{equation*}
Now we apply the Reconstruction Lemma~(\ref{lm:newReconst})
to the 5-point correlators:
\begin{eqnarray*}
  \langle \newx,\newx,\newxsq,\newxsq,\newxsq \rangle &=& \sum_l \frac{2!}{1!1!}\langle
\newx,\newxsq,\newx,\delta_l
\rangle \langle \delta_l',\newx,\newxsq,\newx \rangle \\
   &=& 2\langle \newxsq,\newx,\newx,\newx\rangle \langle 6\newx,\newx,\newx,\newxsq\rangle + 2\langle \newxsq,\newx,\newx,\newy \rangle \langle -2\newy,\newx,\newx,\newxsq\rangle \\
   &=& 12a^2.
\end{eqnarray*}

A similar computation for 
$\langle \newy,\newy,\newxsq,\newxsq,\newxsq \rangle$ gives us
\begin{eqnarray*}
  \langle \newxsq,\newxsq,\newxsq,\newy,\newy \rangle &=& -4\corr{Y,Y,X,X^2}^2 = -36a^2
\end{eqnarray*}
and
\begin{equation*}
  \langle \newxsq,\newxsq,\newxsq,\newx,\newy \rangle = 0.
\end{equation*}

Applying the Reconstruction Lemma to the six-point correlator $\langle
\newxsq,\newxsq,\newxsq,\newxsq,\newx,\newxsq \rangle = -\frac13\langle
\newxsq,\newxsq,\newxsq,\newxsq,\newx,\newy^2 \rangle $ yields
\begin{align*}
 -3   \langle
\newxsq,\newxsq,\newxsq,\newxsq,\newx,\newxsq \rangle
&= \sum_l \frac{3!}{2!1!} \langle \newxsq,\newxsq,\newxsq,\newy,\delta_l \rangle \langle \delta_l',\newx,\newy,\newxsq\rangle \\
&\qquad-\sum_l
\frac{3!}{1!2!}\langle\newxsq,\newxsq,\newxsq,\newx,\newy,\delta_l\rangle\langle \delta_l',\newy,\newy,\newxsq \rangle\\
&= 3 \langle \newxsq,\newxsq,\newxsq,\newy,\newy \rangle \langle -2\newy,\newx,\newy,\newxsq\rangle\\
&\qquad -
3\langle\newxsq,\newxsq,\newxsq,\newx,\newy,\newx\rangle\langle 6\newx,\newy,\newy,\newxsq \rangle\\
   &= (-6)(-36a^2)(3a) - 18(12a^2)(3a)=0,
\end{align*}
and a similar calculation shows that
$\langle \newxsq,\newxsq,\newxsq,\newxsq,\newy,\newxsq \rangle=0$. Thus both the 6-point
correlators are zero.

Now applying reconstruction to the seven-point correlator $\langle
\newxsq,\newxsq,\newxsq,\newxsq,\newxsq,\newxsq,\newxsq \rangle$ gives us
\begin{align*}
   \langle
\newxsq,\newxsq,\newxsq,\newxsq,\newxsq,\newxsq,\newxsq \rangle\\
&= \sum_l \frac{4!}{2!2!}\langle \newxsq,\newxsq,\newxsq,\newx,\delta_l\rangle \langle \delta_l', \newxsq,\newx,\newxsq,\newxsq,\rangle\\
&= \sum_l 6\langle \newxsq,\newxsq,\newxsq,\newx,\newx\rangle \langle 6\newx, \newxsq,\newx,\newxsq,\newxsq,\rangle \\
&=6(12a^2)(6)(12a^2) = (72)^2a^4.
\end{align*}
Inserting these correlators into the potential 
$$\Phi(\bt) 
= \sum_{k\ge3}\sum_{\varkappa_1,\dots,\varkappa_k} \corr{\varkappa_1,\dots,\varkappa_k} \frac{\prod_{i=1}^k t_{\varkappa_i}}{k!}$$
gives the desired result.\end{proof}

\begin{crl}
The potential $\Phi_{D_4}^S$ for the Saito Frobenius manifold of $D_4 = x^3+xy^2$ with primitive form $dx\wedge dy$ is
\medskip
$$\Phi^S_{D_4}(\bt) 
=\frac{1}{12}t_X^2t_{\id}-\frac{1}{4}t_Y^2t_{\id} + \frac{1}{12}t_{X^2}t_{\id}^2 
- \frac{1}{216} t_X^3t_{X^2} - \frac{1}{24} t_Xt_Y^2t_{X^2}
+ \frac{1}{2592} t_X^2t_{X^2}^3 - \frac{1}{864} t_Y^2t_{X^2}^3 
+ \frac{1}{3919140}t_{X^2}^7.
$$
\medskip
\end{crl}
\begin{proof}
In \cite[Prop 6.4.4]{FJR} we computed that for the primitive form $6dx\wedge dy$ we have 
$\corr{\id,Y,Y} = -3$, $\corr{\id,X,X} = \corr{\id,\id,X^2} = 1$, $\corr{X,X,X,X^2}= -\frac{1}{6}$, and $\corr{Y,Y,X,X^2} = -\frac{1}{2}$.  Rescaling the primitive form by a non-zero $\beta$ rescales the entire potential (and hence also the metric) by $\beta$.  The choice of primitive form $dx \wedge dy$, corresponding to $\beta=\frac16$, gives us the same metric and three-point correlators we computed earlier in this paper.  Substituting the value of $a=-\frac{\beta}{6}$ gives the result.
\end{proof}

 We also have the fundamental result of 
    Frenkel-Givental-Milanov. 
    \begin{thm}[\cite{GM,FGM}]\label{thm:FGM}
    For any ADE-singularity $W$, the Saito-Givental (B-side) total descendant potential ${\D}^{SG}_{W}$ is a $\tau$-function of
    the  corresponding Drinfeld-Sokolov/Kac-Wakimoto ADE-hierarchy.
    \end{thm}

Theorems~\ref{thm:allfromfour} and \ref{thm:FGM} will allow us to prove the following lemma.
\begin{lm}\label{lm:nonzerogivesall}
For both the FJRW theory of $D_4^T$ with symmetry group $G_{max}$ and the FJRW theory for $D_4$ with symmetry group $\genj$, if  the correlator $\corr{X,X,X,X^2}$ is non-zero, then the corresponding total descendant potential function ${\D}^{FJRW}_{D_4,J}$ or ${\D}^{FJRW}_{D^T_4,J^T}$ is a tau-function of the $D_4$ Kac-Wakimoto/Drinfeld-Sokolov hierarchy.
\end{lm}
\begin{proof}
For each $\varkappa$ in the standard basis, the change of variables $t_\varkappa \mapsto \sigma^{deg(\varkappa)}t_\varkappa$ rescales the degree-three part of the Saito genus-zero potential (that is, the metric and all three-point correlators) by $\sigma^{\chat_W}$, and it rescales the degree-four part of the potential (the four point correlators) by $\sigma^{\chat_W+1}$, but it leaves the flat identity element $\id$ unchanged.

Re-scaling the primitive form by a non-zero scalar $\beta$ also leaves the identity element $\id$ unchanged, but it rescales the entire potential by the same element $\beta$. If we choose 
$\beta=\sigma^{-\chat_W}$, then the composition of the change of variables followed by rescaling the primitive form leaves the degree-three part of the potential (i.e., the entire Frobenius algebra structure) completely unchanged, and it rescales the  degree-four part of the potential by $\sigma$.

If the correlator $\corr{X,X,X,X^2}$ is non-zero, we can change variables and rescale the primitive form to make the Saito (B-model) correlator $\corr{X,X,X,X^2}$ precisely match the FJRW (A-model) correlator. Since all the the four- and higher-point correlators are completely determined by $\corr{X,X,X,X^2}$, this means that the genus-zero potentials $\Phi^{FJRW}$ and $\Phi^{\text{\emph{Saito}}}$ will match exactly.

By Theorem~\ref{thm:allfromfour}, this shows that the total descendant potential function $\D^{FJRW}_{D_4,\genj}$ or $\D^{FJRW}_{D_4^T,G_{max}}$ will precisely match the Givental-Saito total descendant potential function $\D^{GS}_{D_4}$.  The desired result now follows from Theorem~\ref{thm:FGM}. 
\end{proof}

\section{Proof of the Main Theorem}

In this section we will finish the proof of Theorem~\ref{thm:main}.  The hardest part of this proof boils down to proving that the seven-point correlator $\corr{X^2,X^2,X^2,X^2,X^2,X^2,X^2}$ is non-zero in the FJRW theory for $D_4$ with symmetry group $\genj$.  We do this by a long computation in Lemma~\ref{lm:three}.
\let\oldthethm=\thethm
\renewcommand{\thethm}{\ref{thm:main}}
\begin{thm}
\ 
\begin{enumerate}
\item ${\D}^{FJRW}_{D_4, \genj}$ is a tau-function of the $D_4$ Kac-Wakimoto/Drinfeld-Sokolov hierarchy.
\item ${\D}^{FJRW}_{D^T_4, G_{max}}$ is a tau-function of the $D_4$ Kac-Wakimoto/Drinfeld-Sokolov hierarchy.
\end{enumerate}
\end{thm}
\let\thethm=\oldthethm
\begin{proof}
By Lemma~\ref{lm:nonzerogivesall} all that is required is to show that the correlator $\corr{X,X,X,X^2}$ does not vanish for the FJRW $D^T_4$ and $D_4$ theories.

In \cite[\S6.3.7]{FJR} the correlator $\corr{X,X,X,X^2}$ for the FJRW theory of $D^T_4$ with symmetry group $G_{max}$ is computed to be $\alpha/6^3$, where $\alpha^2=1/6$. This proves the theorem for the case of $D_4^T$. 

As noted in the introduction, computing the correlator $a=\corr{X,X,X,X^2}$ in the $D_4$ case directly is very difficult because three of the insertions are broad.  Computing the correlator in such a case is a difficult PDE problem, and we do not yet have techniques to solve it explicitly.  However, by Theorem~\ref{thm:potential} we know that the seven point correlator $\corr{X^2,X^2,X^2,X^2,X^2,X^2,X^2}$ satisfies $$\corr{X^2,X^2,X^2,X^2,X^2,X^2,X^2} =  (72)^2a^4,$$
so showing that $a$ is non-zero is equivalent to showing that 
$\corr{X^2,X^2,X^2,X^2,X^2,X^2,X^2}$ is non-zero.  This is proved in Lemma~\ref{lm:three}, below. \end{proof}

\begin{lm}\label{lm:three}
In the FJRW (A-model) theory of $D_4$ with symmetry group $\genj$, the genus-zero seven-point correlator $\corr{
\newxsq,\newxsq,\newxsq,\newxsq,\newxsq,\newxsq,\newxsq}$ is non-zero.
\end{lm}
\begin{proof}

This correlator has only narrow insertions and has $\deg(|\LL_x|)=\deg(|\LL_y|)=-3$; 
to verify that it is a concave correlator we must check that $\pi_*\left(\LL_x \oplus \LL_y \right) = 0$.  To do this we verify that the degree of $\LL_x$ and $\LL_y$ is negative on each irreducible component of each curve $\cC$ in 
$\W_{0,7}(X^2, X^2, X^2, X^2, X^2, X^2, X^2)$.  We examine all possible degeneration of a genus-zero seven-point graph 
$\Gamma \in \bGa_{0,7,D_4}(J^2, J^2, J^2, J^2, J^2, J^2, J^2)$.  
The single-edge degenerations are shown in Figure \ref{fig:seven-two-five}.  Checking all of these line bundle degrees is a straightforward but tedious computation done using the code described below in Remark \ref{rem:code}.
The correlator is concave and can be computed directly using the concavity axiom \cite[Axiom~5.a]{FJR}.  Because $\LL_x$ and $\LL_y$ are concave, the pushforwards satisfy $-R^{\bullet}\pi_* (\LL_x) = R^{1}\pi_* (\LL_x)$  and $-R^{\bullet}\pi_* (\LL_y) = R^{1}\pi_* (\LL_y)$ and are vector bundles.  By Riemann-Roch the (complex) rank of each of these push-forward bundles is $2$.
By the concavity axiom we have
\begin{align}
\Lambda^{D_4}_{0,7}(\newxsq,\newxsq,\newxsq,\newxsq,\newxsq,\newxsq,\newxsq) &= \frac{1}{\deg(\st)}c_{4}\left(R^{1}\pi_* (\LL_x)\oplus R^{1}\pi_* (\LL_y)\right)\notag\\
&=\frac1{\deg(\st)}c_2\left(R^{1}\pi_* (\LL_x)\right)c_2\left(R^{1}\pi_* (\LL_y)\right).
\label{eq:c2Lambda}
\end{align}

In order to compute these Chern classes, we will use Chiodo's formula \cite[Thm 1.1.1]{Chi} for the Chern character of the pushforward of an $r$th root of $K_{\log} $.  To do this, we note first that  the stack $\W_{g,k,{\genj}}(D_4)$ can be identified with the open and closed substack of $\W_{g,k,G_{max}}(D_4)$ corresponding to $(D_4 + x^2y)$-curves (see \cite[\S2.3]{FJR}).  That means that $\LL_x^{\otimes 3} \cong K_{\log}$, and  $\LL_x \cong K_{\log}\otimes\LL^{\otimes -2}_y$ and $\LL_y \cong K_{\log}\otimes\LL^{\otimes -2}_x$.  This implies that $$\LL_x \cong \LL_y \dsand \LL_x^3 = K_{\log}.$$
This shows that 
the stack $\W_{g,k,{\genj}}(D_4)$ is canonically isomorphic to the stack $\MM^{1/3}_{g,k}$ of third roots of $K_{\log}$. 

This stack is endowed with a tautological ring, similar to that of $\MM_{g,k}$.  We briefly mention a few special cohomology classes now.  We use the $\psi$, $\kappa$, and $\Delta$ classes as defined in  \cite[Section 2.4]{FJR}.  Note that the $\psi$ classes on $\W_{g,k}$ are the pullbacks of the corresponding $\psi$ classes on $\MM_{g,k}$, the $\kappa$ classes on $\W_{g,k}$ are the pullbacks of the  $\kappa$ classes on $\MM_{g,k}$.

Chiodo's formula states that for the universal $r$th root $\LL$ of the sheaf $K_{\log}^{ s} $ on  the universal family of pointed orbicurves $\pi: \cC\rTo \MM^{s/r}_{g,k}(\ga_1,\dots,\ga_k)$  and with local group $\langle \ga_j\rangle$ at a marked point $p_j$  acting as $\exp(2\pi i \Theta^{\gamma_j})$ on $\LL$  
we have
$$\Chern(R^{\bullet}\pi_* (\LL)) =
\sum_{d\ge 0} \left[\frac{B_{d+1}(s/r)}{(d+1)!}\kappa_d - \sum_{i=1}^k\frac{B_{d+1}(\Theta^{\ga_i})}{(d+1)!}\psi_i^d + \frac12\sum_{\Gac} \frac{rB_{d+1}(\Theta^{\ga_-})}{(d+1)!}\rhot_{\Gac *}\left(\sum_{\substack{i+j=d-1\\i,j\ge0}} (-\psi_+)^i\psi_{-}^j\right)\right],
$$
where the second sum is taken over all decorated stable graphs $\Gac$ with one pair of tails labeled $+$ and $-$, respectively, so that once the $+$ and $-$ edges have been glued, we get a single-edged, $n$-pointed, connected, decorated graph of genus $g$ and with additional decoration ($\ga_+$ and $\ga_-$) on the internal edge.    Each such graph $\Gac$ has the two cut edges, decorated with group elements $\ga_+$ and $\ga_-$, respectively, and the map $\rhot_{\Gac}$ is the corresponding gluing map  $\MM^{r/s}_{\Gac}\rTo \MM^{r/s}_{g,k}(\ga_1,\dots,\ga_k)$.

Notice that our version of the formula looks slightly different from that in  \cite[Thm 1.1.1]{Chi} because we have used orbifold line bundles which are roots of $K_{\log}$ on a stable orbicurve $\cC$, while Chiodo used invertible sheaves which are roots of the sheaf $\omega_{\log}$ on a stable curve (with no orbifold structure).   It is shown in \cite[\S2.1]{FJR} that if the local group acts by $\exp(2\pi i \Theta^{\gamma_j})$ on $\LL$ with $\Theta^{\gamma_j} = a_j/r$,  then the corresponding invertible sheaf of sections on the underlying stable curve (without orbifold structure) is an $r$th root of the sheaf $\omega_{\log}(-\sum_{j=1}^k a_j p_j)$---matching the part of Chiodo's formula corresponding to the classes $\psi_j$.   But at the nodes, Chiodo's original formula is written in terms of group actions on an invertible sheaf---that is, in terms of orbifold structure on the node.  In this case, if the local group $\langle\gamma_+\rangle$  acts on a line bundle  as $\exp(2\pi i \Theta^{\gamma_+})$ then it acts on the sheaf of sections as $\exp(-2\pi i \Theta^{\gamma_+}) = \exp(2\pi i \Theta^{\gamma_-})$.
This is why the formula here has $\Theta^{\ga_-}$ in the last sum.

Note that we have  $\LL_x\cong \LL_y$, and we can use Chiodo's formula for both $\LL_x$ and $\LL_y$.  We have,  $r =3$, $s=1$, and $\Theta^{\ga_i} = 2/3$ for every $i\in\{1,\dots,7\}$.  Expressing $c_2$ in terms of the Chern character, we have

\begin{align*}
c_2\left(R^{1}\pi_* (\LL)\right) & = c_2\left(-R^{\bullet}\pi_* (\LL)\right) \notag\\
 &= c_1^2(R^{\bullet}\pi_* \LL) -c_2(R^{\bullet}\pi_* \LL)\notag\\
 &= \Chern_1^2(R^{\bullet}\pi_* \LL)-\frac12\Chern_1^2(R^{\bullet}\pi_* \LL) + \Chern_2(R^{\bullet}\pi_* \LL)\notag\\
 &= \frac12\Chern_1^2(R^{\bullet}\pi_* \LL) + \Chern_2(R^{\bullet}\pi_* \LL)\notag\\
&= \frac18\left(B_2(1/3) \kappa_1 - \sum_{i=1}^7 B_2(2/3)\psi_i + \frac12\sum_{\Gac} r B_2(\Theta^{\ga_-})\rhot_{\Gac *}(1)\right)^2 
\notag\\
& \qquad \qquad + \frac16\left(B_3(1/3) \kappa_2 - \sum_{i=1}^7 B_3(2/3)\psi_i^2  
+\frac12\sum_{\Gac} r B_3(\Theta^{\ga_-})\rhot_{\Gac *}(\psi_- - \psi_+)
\right)\notag\\
&= \frac18\left(-\frac1{18}\kappa_1 + \sum_{i=1}^7 \frac1{18}\psi_i + \frac12\sum_{\Gac} 3 B_2(\Theta^{\ga_-})\Deltat_{\Ga}\right)^2 
\notag\\
& \qquad \qquad + \frac16\left(\frac1{27} \kappa_2 + \sum_{i=1}^7 \frac1{27}\psi_i^2  
+\frac12\sum_{\Gac} 3 B_3(\Theta^{\ga_-})\rhot_{\Gac *}(\psi_- - \psi_+)
\right),
\end{align*}
where $\Deltat_{\Ga}$ is the boundary divisor in $\MM^{1/3}_{0,7}$ corresponding to $\rhot_{\Gac}(1)$.

We now push forward from $\MM^{1/3}_{0,7}$ to the stack of stable curves $\MM_{0,7}$.  
Recall that the $\kappa_d$ and the $\psi^d$ are pullbacks of their counterparts on $\MM_{0,7}$,
 and since the morphism $\st:\MM^{1/3}_{0,7} \rTo \MM_{0,7}$ has three-fold ramification along the locus corresponding to $\Gamma$, we have $3\rhot_{\Gac *}(\psi_{\pm}^{j})  = \st^*(\rho_{|\Gac| *}(\psi_{\pm}^{j}) )$ and $3\Deltat_\Ga :=3\rhot_{\Gac *}(1) = \rho_{|\Gac| *}(1) =:\Delta_\Ga$, where $|\Gac|$ is the undecorated graph underlying $\Gac$, and $\rho_{|\Gac|}:\MM_{|\Gac|} \rTo \MM_{0,7}$ is the associated gluing map.

Applying Equation~(\ref{eq:c2Lambda}) we have  
\begin{align*}
\Lambda^{D_4}_{0,7}(\newxsq,\newxsq,& \newxsq,\newxsq,\newxsq,\newxsq,\newxsq)\\ 
&= \Biggl[
\frac18\left(-\frac1{18}\kappa_1 + \sum_{i=1}^7 \frac1{18}\psi_i + \frac12\sum_{\Gac}  B_2(\Theta^{\ga_-})\Delta_{\Ga}\right)^2 
\notag\\
& \qquad \qquad+ \frac16\Biggl(\frac1{27} \kappa_2 + \sum_{i=1}^7 \frac1{27}\psi_i^2  
+\frac12\sum_{\Gac}  B_3(\Theta^{\ga_-})\rho_{\Gac *}(\psi_- - \psi_+)
\Biggr)
\Biggr]^2.\notag
\end{align*}

\begin{figure}
\begin{tikzpicture}

\node at (0,0){\includegraphics{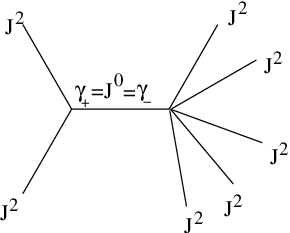}};
\node at (0,-2.3){$\deg(|\LL_x|)=\deg(|\LL_y|) = -3$};
\node at (0,-5){\includegraphics{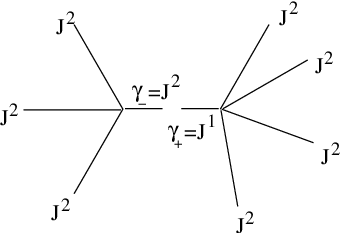}};
\node at (-3,-7.3){$\deg(|\LL_x^-|)=\deg(|\LL_y^-|) = -2$};
\node at (3,-7.3){$\deg(|\LL_x^-|)=\deg(|\LL_y^-|) = -2$};
\node at (0,-10){\includegraphics{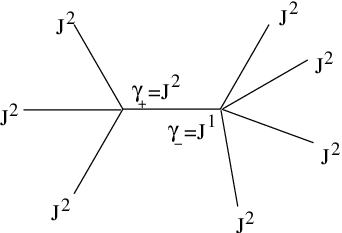}};
\node at (-3,-12.3){$\deg(|\LL_x^-|)=\deg(|\LL_y^-|) = -2$};
\node at (3,-12.3){$\deg(|\LL_x^-|)=\deg(|\LL_y^-|) = -2$};\end{tikzpicture}
\caption{\label{fig:seven-two-five} The three forms that a one-edge degeneration can take for a seven-point graph $\Ga \in \bGa_{0,7,D_4}(J^2 , J^2,J^2,J^2,J^2,J^2,J^2)$.  The first corresponds to both $|I|=2$ and $|I|=5$, the second corresponds to $|I|=4$ and the third to $|I|=3$.}
\end{figure}
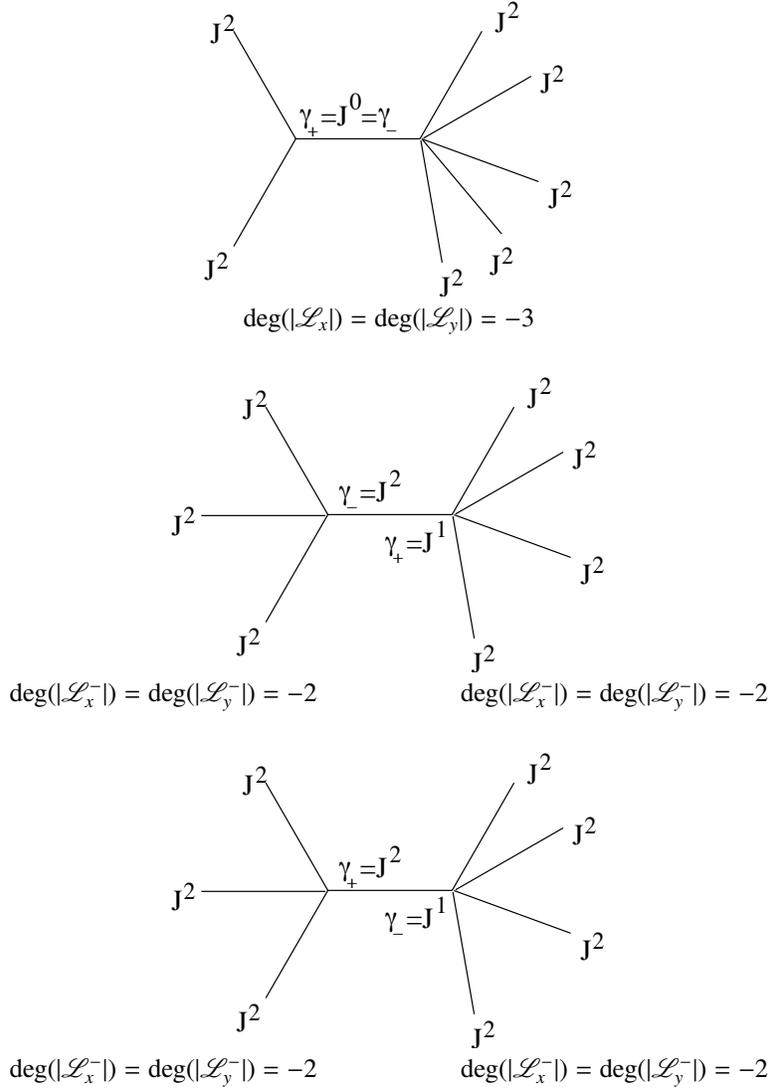

We can use either Carel Faber's MAPLE code \cite{FaMgnF} or Drew Johnson's SAGE code \cite{Jo} to complete the calculation.  For simplicity, and to match standard usage, we index the graphs $\Gac$ by the subsets $I\subset [7]:=\{1,\dots,7\}$ such that $1\in I$ and $2 \le |I| \le 5$.  That is, $\Delta_I$ is the same as $\Delta_{\Gac}$ where the graph $\Gac$ has its tails on one vertex labeled by the elements of $I$ and the symbol ``$+$," whereas the tails on the second vertex are labeled by the elements of $I^c = [7]\setminus I$ and the symbol ``$-$."  This indexing includes exactly half of all the graphs---the other half come from swapping the positions of the $+$ and $-$ labels.
If we denote by $\Gacp$ the graph obtained from  $\Gac$ by swapping the $+$ and $-$ labels, then it is easy to see that the following hold:
\begin{align*}
\Delta_{\Gac} &= \Delta_{\Gacp}\\ 
\rho_{\Gac *} (\psi_{+}) &=  \rho_{\Gacp *} \psi_{-}\\
\rho_{\Gac *} (\psi_{-}) &=  \rho_{\Gacp *} \psi_{+}.
\end{align*}
Moreover, since for any decorated graph we have $\ga_+^{-1} = \ga_{-}$, we have either $\Theta^{\ga_-} =1- \Theta^{\ga_+}$ if $\Theta_{\ga_+}\neq 0$, or  $\Theta^{\ga_-} = \Theta^{\ga_+} = 0$ otherwise. And the Bernoulli numbers $B_n(t)$ satisfy the well-known relation
$$B_n(1-t) = (-1)^n B_n(t).$$ 
Therefore, we have 
\begin{align}\label{eq:LambdaI}
\Lambda^{D_4}_{0,7}(\newxsq,\newxsq,& \newxsq,\newxsq,\newxsq,\newxsq,\newxsq)\\ 
&= \Biggl[
\frac18\left(-\frac1{18}\kappa_1 + \sum_{i=1}^7 \frac1{18}\psi_i + \sum_{1\in I\subset [7]}  B_2(\Theta^{\ga_-})\Delta_{I}\right)^2 
\notag\\
& \qquad \qquad+ \frac16\Biggl(\frac1{27} \kappa_2 + \sum_{i=1}^7 \frac1{27}\psi_i^2  
+\sum_{1\in I \subset[7]}  B_3(\Theta^{\ga_-})\rho_{I *}(\psi_- - \psi_+)
\Biggr)
\Biggr]^2.\notag
\end{align}

We now need to compute the values of $\ga_-$ that occur for the various choices of $I$.  The graph
$\Gac$ must be one of the three forms depicted in Figure~\ref{fig:seven-two-five}. There are ${7 \choose 2} = 21$ choices of ways to label the external edges of the first graph and ${7 \choose 3} = 35$ choices for the second and third.
In the first case, we have $B_2(0) = 1/6$ and $B_3(0) = 0$. For the second type, we have
$B_2(2/3) = -1/18$ and $B_3(2/3) = -1/27$,  while in the last case we have 
$B_2(1/3) = -1/18$ and $B_3(1/3) = 1/27$. 

Faber and Johnson's codes compute intersections of divisors on the stack of stable curves, so we must express every one of the classes in Equation~(\ref{eq:LambdaI}) in terms of divisor classes.  The only classes that are not already products of divisors are $\kappa_2$ and $\rho_{I *}\psi_{\pm}$.  To rewrite the first, we use the fact that on $\MM_{0,5}$ we have
$$\kappa_{(0,5),2} = \kappa_{(0,5),1}\Delta_{1,2,3}.$$
This follows, for example, from \cite[Cor 2.2]{KaKi}.   Moreover, $\kappa_2$, $\psi_i$ and $\Delta_I$ pull back (see \cite[Lm 1.2]{AC})
along the forgetting tails map $\tau:\MM_{0,n+1} \rTo \MM_{0,n}$ as 
\begin{align*}\label{eq:taukappa}
\tau^*(\kappa_{(0,n),a})&= \kappa_{(0,n+1),a} - \psi_{n+1}^a\\
\tau^*(\psi_i) &= \psi_i - \Delta_{i,n+1}\\
\tau^*(\Delta_I) &= \Delta_{I} + \Delta_{I\cup \{n+1\}}. 
\end{align*}

Using the pullback twice, we get
\begin{align*}
\kappa_{(0,7),2} 
&= \tau^*(\kappa_{(0,6),2}) + \psi_7^2\\
&= \tau^*(\tau^*(\kappa_{(0,5),2})+\psi_6^2) + \psi_7^2\notag\\
&= \tau^*(\tau^*(\kappa_{(0,5),1}\Delta_{1,2,3}) +\psi_6^2) + \psi_7^2\notag\\
&= \tau^*\left((\kappa_{(0,6),1}-\psi_6)(\Delta_{1,2,3} + \Delta_{1,2,3,6})  +\psi_6^2\right) + \psi_7^2\notag \\
&= (\kappa_{(0,7),1}-\psi_7-(\psi_6-\Delta_{6,7}))(\Delta_{1,2,3} +\Delta_{1,2,3,7}+ \Delta_{1,2,3,6}+\Delta_{1,2,3,6,7})  +(\psi_6-\Delta_{6,7})^2 + \psi_7^2. \notag
\end{align*}

For ease of computation, we further rewrite the boundary divisors in this equation so that they all have $1$ in their index set.  This can easily be done by means of the obvious equality
$$\Delta_{I} = \Delta_{I^c}.$$

It is more messy to rewrite the classes $\rho_{I *}(\psi_{\pm})$ in terms of divisors.  To do this, we use the following well-known relation:
for any distinct $a,b,i \in \{1,\dots,k\}$, we have 
\begin{equation*}
\psi_i  = \sum_{\substack{i \in I \\ a,b \nin I}} \Delta_I.
\end{equation*}
And a straightforward computation now yields the following lemma.
\begin{sublm}
For any subset $K$ with $1\in K$ and for any distinct $r,s \in K\setminus \{1\}$   and any distinct $t,u \nin K$ we have
\begin{align*}
\rho_{K *}(\psi_+) &= 0 \text{ if $|K|\le 2$}\\
\rho_{K *}(\psi_+) &= \sum_{\{1,r,s\} \subseteq I\subset K} \Delta_K \Delta_I + \sum_{1\in I \subseteq K \setminus \{r,s\}} \Delta_K \Delta_{I \cup K^c}\\
\rho_{K *}(\psi_-) &= 0 \text{ if $|K|\ge 5$}\\
\rho_{K *}(\psi_-) &= \sum_{\substack{\emptyset \neq I \subset K^c \\ t,u \nin I}}
\Delta_K \Delta_{I \cup K}.
\end{align*}
\end{sublm}

To complete the computation, we make all these substitutions into Equation~(\ref{eq:LambdaI}) and then use either Drew Johnson's SAGE code \cite{Jo} or Carel Faber's MAPLE code \cite{FaMgnF} to integrate the resulting (enormous) sum of degree-four monomials of divisors.  Both the SAGE and the MAPLE computations yield 
$$\langle \newxsq,\newxsq, \newxsq,\newxsq,\newxsq,\newxsq,\newxsq\rangle^{D_4}_{0} =   
\frac{2}{27},$$
which is not zero, as desired.  \end{proof}

\begin{rem}\label{rem:code}
The code which we used to make our computations is available at \url{http://math.byu.edu/~jarvis/D4Code.html}.  We made the computations both in SAGE \cite{sage} and in MAPLE.  The SAGE computations depend on code \cite{Jo} written by Drew Johnson  for computing intersection numbers of classes on $\overline{M}_{g,n}$, while the MAPLE version depends on code  \cite{FaMgnF} written by Carel Faber for computing intersection numbers of  classes on $\overline{M}_{g,n}$.  
\end{rem}

\bibliographystyle{amsplain}

\providecommand{\bysame}{\leavevmode\hbox
to3em{\hrulefill}\thinspace}

\end{document}